\documentclass[options]{asl}

\usepackage{stmaryrd}
\usepackage{soul}

\title{Rank and Randomness}
\author{Rupert H\"olzl and Christopher P.\ Porter}

\newpage

\newtheorem{theorem}{Theorem}[section]
\newtheorem{lemma}[theorem]{Lemma}

\newtheorem{question}[theorem]{Question}

\theoremstyle{remark}

\newtheorem{rmks}{Remarks}[section]
\newtheorem*{Examples}{Examples}

\renewcommand{\L}{\mathcal{L}}

\renewcommand{\O}{\mathcal{O}}
\renewcommand{\P}{\mathcal{P}}
\newcommand{\Q}{\mathcal{Q}}

\newcommand{\calS}{\mathcal{S}}

\newcommand{\U}{\mathcal{U}}

\newcommand{\MLR}{\mathsf{MLR}}

\newcommand{\dom}{\text{dom}}

\newcommand{\llb}{\llbracket}
\newcommand{\rrb}{\rrbracket}

\newcommand{\cs}{2^\omega}
\newcommand{\uh}{{\upharpoonright}}
\renewcommand{\phi}{\varphi}
\newcommand{\str}{2^{<\omega}}

\newcommand{\rk}{\mathit{rk}}
\renewcommand{\tt}{\mathit{tt}}

\newcommand{\atom}{\mathsf{Atoms}}

\newcommand{\claimqed}{\hfill\quad$\triangle$}

\begin{document}
\sloppy

\begin{abstract}
We show that for each computable ordinal $\alpha>0$ it is possible to find  in each Martin-L\"of random $\Delta^0_2$~degree a sequence~$R$ of Cantor-Bendixson rank $\alpha$, while ensuring that the sequences that inductively witness $R$'s~rank are all Martin-L\"of random with respect to a single countably supported and computable measure. This is a strengthening for random degrees of a recent result of Downey, Wu, and Yang, and can be understood as a randomized version of it.

\end{abstract}

\maketitle

\section{Introduction}\label{intro}

The notion of a random sequence, which was originally defined by Martin-L\"of for spaces endowed with the Lebesgue measure, can also be studied in spaces endowed with other computable measures. One particular instance is given by sequences that are random with respect to a computable, countably supported measure $\mu$ on $\cs$, that~is, a computable measure for which there is some countable collection $(X_i)_{i\in\omega}$ such that ${\mu\bigl(\bigcup_{i\in\omega}\{X_i\}\bigr)=1}$. The behavior of such sequences has been studied by Bienvenu and Porter~\cite{BiePor12}, Porter~\cite{Por15}, as well as H\"olzl and Porter~\cite{HolPor16}.

Kautz~\cite{Kau91} showed that any sequence $X$ that satisfies $\mu(\{X\})>0$ for some computable measure~$\mu$ is itself computable. Consequently, one might expect that randomness with respect to a computable, countably supported measure is a trivial notion; and in fact, such measures are referred to as \emph{trivial measures} in the above articles, following terminology of Kautz. But the terminology is misleading, as in fact there are computable, countably supported measures $\mu$ for which there exist other sequences $X$ that are  random  with respect to $\mu$, besides those satisfying $\mu(\{X\})>0$.
As shown in the above studies, these sequences have exotic properties, such as having extremely slow-growing initial segment complexity or the oracle power to compute fast-growing functions.

Following the terminology of Levin and Zvonkin~\cite{ZvoLev70}, we refer to sequences that are random with respect to a computable measure as \emph{proper}. Bienvenu and Porter~\cite{BiePor12} constructed the first example of a proper sequence of Cantor-Bendixson rank $1$;\footnote{For the expert reader, we mention already at this point that if a sequence~$X$ of Cantor-Bendixson rank $1$ is random with respect to a computable measure~$\mu$ then $\mu$ must have atoms but $X$ cannot be one of them.} and Porter~\cite{Por15} applied the same technique to show that in every $\Delta^0_2$ random Turing degree (that is, a $\Delta^0_2$ degree containing a random sequence), there is a proper sequence of Cantor-Bendixson rank $2$.  These results naturally lead to the following questions:
\begin{itemize}
\item[--] For each computable ordinal $\alpha$, is there a proper sequence $X$ of rank~$\alpha$?
\item[--] Moreover, can such a proper sequence be found in each $\Delta^0_2$ random degree?
\end{itemize}

In this article, we answer both questions in the affirmative.  In fact, we prove something significantly stronger, namely that we can find a computable, countably supported measure~$\mu$ such that the desired $X$ is Martin-L\"of random with respect to $\mu$ and such that we can find a collection of $\mu$-random sequences of rank lower than that of $X$ that witness the rank of $X$. The details of this strong property will be discussed in Section~\ref{subsec-murank}.

Our main result is the following.
\begin{theorem}\label{thm-main}
 For each $\Delta^0_2$ random degree $\mathbf{r}$ and each computable ordinal $\alpha>0$, there is a countably supported computable measure $\mu$ and a sequence $R\in\mathbf{r}$ such that
\begin{itemize}
\item[(i)] $R$ is random with respect to $\mu$, 
\item[(ii)] $R$ has Cantor-Bendixson rank $\alpha$, and
\item[(iii)] the support of $\mu$ is rank-faithful.
\end{itemize}
\end{theorem}

We note that Theorem \ref{thm-main} is related to a recent result of Downey, Wu, and Yang~\cite{DowWuYan15} who showed that for every $\Delta^0_2$ degree $\mathbf{a}$ and every computable ordinal~$\alpha>0$, there is some~${A\in\mathbf{a}}$ with Cantor-Bendixson rank $\alpha$.    This latter result generalizes several older ones, namely that (a)~for every computable ordinal~$\alpha>0$, there is some $\Delta^0_2$ degree $\mathbf{a}$ that contains a set of Cantor-Bendixson rank $\alpha$ (due to Cenzer and Smith~\cite{CenSmi89}), (b)~every $\Delta^0_2$~degree contains a rank $1$ point (ibid.), and (c)~for every c.e.\ degree~$\mathbf{c}$ and every computable ordinal $\alpha>0$, there is a c.e.\ set $C\in\mathbf{c}$ of Cantor-Bendixson rank~$\alpha$ (due to Cholak and Downey~\cite{ChoDow93}). Our result can thus be viewed as a ``randomized" version of the above result of Downey, Wu, and Yang.

Before we turn to the proof of Theorem \ref{thm-main}, we review the relevant background in Section~\ref{sec_background} and discuss the special cases of Theorem \ref{thm-main} where $\alpha$ is finite and~${\alpha=\omega}$ in Section~\ref{sec_special_case}.  We then turn to the full proof
in Section~\ref{sec_full_proof}.

\section{Background}\label{sec_background}

\subsection{Notation}

For $A,B\in\cs$, $A\oplus B$ is the computable join, where ${(A\oplus B)(2n)=A(n)}$ and $(A\oplus B)(2n+1)=B(n)$ for $n\in\omega$.  Moreover, for~$n\geq 1$ and $A_0,\dotsc, A_n\in\cs$, we define~$\bigoplus_{i=0}^nA_i$ recursively by
\[
\bigoplus_{i=0}^nA_i=\Biggl(\bigoplus_{i=0}^{n-1}A_i\Biggr)\oplus A_n
\]
For $Y,Z\in\cs$, let $Y\uh_Z$ be the sequence that satisfies
$
Y\uh_Z(n)=Y(p_Z(n)),
$
where $p_Z$ is the principal function of $Z$, that is, $p_Z(n)$ is the $(n+1)^{\mathrm{st}}$ element of $Z$ in increasing order.

\subsection{Randomness with respect to a computable measure}

Recall that a probability measure on $\cs$ is determined by its values on sets of the form $\llb\sigma\rrb=\{X\in\cs\colon\sigma\prec X\}$, where $\prec$ is the initial segment relation.  A measure~$\mu$ on~$\cs$ is thus called computable if the function $\sigma\mapsto\mu(\llb\sigma\rrb)$ is computable as a real-valued function.  Hereafter we will write $\mu(\llb\sigma\rrb)$ as $\mu(\sigma)$.  The Lebesgue measure is denoted by $\lambda$.

We will take Turing functionals $\Phi\colon \cs\rightarrow\cs$ to be defined in terms of pairs of finite strings, with the condition that comparable input strings map to comparable output strings; such a map can be extended to $\cs$ in the natural way.  Computable measures are precisely those measures that are induced by almost total Turing functionals, where a Turing functional~$\Phi\colon \cs\rightarrow\cs$ is almost total if $\lambda(\dom(\Phi))=1$.  Given an almost total Turing functional $\Phi$, the measure induced by $\Phi$, written $\lambda_\Phi$, is defined by
$
\lambda_\Phi(\sigma)=\lambda(\Phi^{-1}(\llb\sigma\rrb))
$.
Here we will focus exclusively on computable measures that are induced by \emph{total} Turing functionals.

Recall that for a computable measure $\mu$, a $\mu$-Martin-L\"of test is a sequence $(\U_i)_{i\in\omega}$ of uniformly effectively open subsets of $\cs$ such that $\mu(\U_i)\leq 2^{-i}$ for every $i\in\omega$, and that a sequence $X\in\cs$ is  {\em random with respect to} $\mu$, written $X\in\MLR_\mu$, if $X\notin\bigcap_{i\in\omega}\U_i$ for every $\mu$-Martin-L\"of test.  It is well-known that for each computable measure $\mu$, there is a \emph{universal} $\mu$-Martin-L\"of test, that is, a single $\mu$-Martin-L\"of test $(\widehat{\U}_i)_{i\in\omega}$ such that $X\in\MLR_\mu$ if and only if $X\notin\bigcap_{i\in\omega}\widehat{\U}_i$.
Clearly, $\MLR_\mu$ is contained in 
$\{X\in2^\omega\colon\forall n\; \mu(X\uh n)>0\}$, the \emph{support} of~$\mu$.

Martin-L\"of randomness can be relativized to an oracle $A$ in a straightforward manner:  we simply replace $(\U_i)_{i\in\omega}$ in the definition of randomness by an $A$\nobreakdash-computable sequence $(\U^A_i)_{i\in\omega}$ of $A$-effectively open subsets of $\cs$.  An important result concerning relative randomness is van Lambalgen's theorem, according to which, for $A,B\in\cs$, $A\oplus B$ is random if and only if $A$ is random relative to $B$ and $B$ is random.

If $\Phi$ is an almost total Turing functional such that $\mu=\lambda_\Phi$, one can show that ${\Phi(\MLR)=\MLR_\mu}$ (see, for instance, Bienvenu and Porter~\cite{BiePor12}).  The inclusion $\subseteq$ is sometimes referred to as the \emph{preservation of randomness}, whereas the inclusion~$\supseteq$ is sometimes referred to as the \emph{no-randomness-from-nonrandomness principle}.

The proof of Theorem \ref{thm-main} will draw on the fact that there are $\Delta^0_2$ random sequences.  Recall that a sequence $A\in\cs$ is $\Delta^0_2$ if there is a uniformly computable sequence of finite sets $(A_s)_{s\in\omega}$ (called a $\Delta^0_2$-approximation of $A$) such that
$
\lim_{s\rightarrow\infty}A_s(n)=A(n)
$
for every $n\in\omega$.

For a measure $\mu$, let $\atom_\mu=\{X\in\cs\colon \mu(\{X\})>0\}$ be the set of {\em atoms} of~$\mu$. A measure without atoms is called {\em continuous}. For a computable measure~$\mu$, the following facts are straightforward to establish:
\begin{itemize}
\item[--] (Kautz~\cite{Kau91}) If $X\in\atom_\mu$, then $X$ is computable.
\item[--] If $X\in\atom_\mu$, then $X\in\MLR_\mu$.
\item[--] If $X$ is computable and $\mu(\{X\})=0$, then $X\notin\MLR_\mu$.
\end{itemize}

\subsection{Cantor-Bendixson rank}

Recall that the \emph{Cantor-Bendixson derivative}~$D(\mathcal{P})$ of a set~$\mathcal{P}$ is the set of nonisolated points in $\mathcal{P}$.
We can iterate the Cantor-Bendixson derivative of~$\mathcal{P}$ as follows:
\begin{itemize}
\item[--] $D^0(\mathcal{P})=\mathcal{P}$;
\item[--] $D^{\alpha+1}(\mathcal{P})=D(D^\alpha(\mathcal{P}))$ for any ordinal $\alpha$; and
\item[--] $D^{\kappa}(\mathcal{P})=\bigcap_{\alpha<\kappa}D^\alpha(\mathcal{P})$ for any limit ordinal $\kappa$. 
\end{itemize}
The \emph{Cantor-Bendixson rank} of a closed set $\mathcal{P}$, denoted $\rk(\mathcal{P})$, is the least ordinal~$\alpha$ such that $D^{\alpha+1}(\mathcal{P}) = D^\alpha(\mathcal{P})$.

We are interested here in the notion of Cantor-Bendixson rank in the context of $\Pi^0_1$ classes, that is, effectively closed subsets of $\cs$.
$X\in\cs$ is \emph{ranked} if there is a $\Pi^0_1$~class $\P$ such that $X\in D^\alpha(\P)\setminus D^{\alpha+1}(\P)$ for some ordinal $\alpha$, and the \emph{Cantor-Bendixson rank} of $X$ in $\P$, denoted~$\rk_\P(X)$, is the least $\alpha$ such that  $X\in D^\alpha(\P)\setminus D^{\alpha+1}(\P)$. For ranked~$X\in\cs$, $\rk(X)$~is the least $\alpha$ such that  $\rk_\P(X)=\alpha$ for some $\Pi^0_1$ class $\P$.  Kreisel \cite{Kre59} proved that if $\P$ is a $\Pi^0_1$~class, then $\rk(\P)$ is less than or equal to $\omega_1^{\mathrm{CK}}$, the least non-computable ordinal, from which it follows that for every ranked $X\in\cs$, $\rk(X)<\omega_1^{\mathrm{CK}}$. 
Lastly, we say that a $\Pi^0_1$~class $\P$ is \emph{rank-faithful} if for all ranked $X\in\P$, $\rk(X)=\rk_\P(X)$.

Note that if a sequence $X$ is ranked, say $\rk(X)=\alpha$ for some $\alpha$, then there are infinitely many distinct sequences $Y_i$, for ${i\in\omega}$, that branch off of~$X$ such that,
\begin{itemize}
\item[--] in the case that $\alpha=\beta+1$, $\rk(Y_i)=\beta$, and
\item[--] in the case that $\alpha$ is a limit, ${\rk(Y_i)<\rk(X)}$ and $\sup_{i\in\omega}\rk(Y_i)=\alpha$.
\end{itemize}
We will informally say that $(Y_i)_{i\in\omega}$ {\em witnesses} the rank of $X$.\label{sdjhsdfgjsdgrtwerwer}  Note further that $\rk(X)=0$ if and only if $X$ is computable (as, on one hand, every isolated point in a $\Pi^0_1$ class is computable; and, on the other hand, $\{X\}$ is a $\Pi^0_1$~class for every computable sequence~$X$).
For more details on rank and $\Pi^0_1$ classes, see Cenzer~\cite{Cen99}.

The following result will be useful in the proof of Theorem \ref{thm-main}.
\begin{lemma}[Cenzer~\cite{Cen99}]\label{lem-rk}
For $A,B\in\cs$, if $A\leq_{\mathrm{tt}} B$ and $B$ is ranked, then $A$ is ranked and 
$
{\rk(A)\leq\rk(B)}$.
If furthermore $A\equiv_{\mathrm{tt}}B$ holds, then ${\rk(A)=\rk(B)}$.
\end{lemma}

 We point out that 
the measure $\mu$ constructed in the proof of Theorem~\ref{thm-main} must necessarily have atoms.  For, as shown by the authors in previous work~\cite{HolPor16}, if a sequence $X$ is random with respect to a computable, continuous measure, then it must be complex, that~is, there must be a computable, unbounded, non-decreasing function $f(n)$ such that $K(X\uh n)\geq f(n)$ for all $n\in\omega$, where $K(\cdot)$~denotes prefix-free Kolmogorov complexity.  As shown by Binns~\cite{Bin08}, every $\Pi^0_1$~class containing a complex element is perfect.\footnote{In fact, Binns~\cite{Bin08} showed that every such class is \emph{computably} perfect, a stronger property.} Thus, no sequence that is random with respect to a computable, continuous measure is ranked.

\subsection{Implications of rank-faithfulness}\label{subsec-murank}

We now explain the significance of the third condition in Theorem~\ref{thm-main}. 
Let $\mu$ be a computable measure with rank-faithful support $\P$.
Let $(\U_i)_{i\in\omega}$ be a universal $\mu$-Martin-L\"of test and write $\mathcal{K}_i=2^\omega\setminus \U_i$ for all $i\in\omega$. Clearly $\mathcal{K}_i\subseteq\P$ for every $i\in\omega$, and since it can only be easier to isolate a path in a subset, we have ${\rk_{\mathcal{K}_i}(R)\leq \rk_\P(R)}$ for every~$i\in\omega$ and $R\in\mathcal{K}_i$.
But since $\rk(R)=\rk_\P(R)$ by the rank-faithfulness of $\P$, we must then have $\rk_{\mathcal{K}_i}(R)= \rk(R)$ for every $i\in\omega$ such that $R\in\mathcal{K}_i$. In other words, $R$'s Cantor-Bendixson rank $\rk(R)$ can already be observed inside $\mathcal{K}_i$.

Now since $\mathcal{K}_i$ contains only $\mu$-random sequences, we have found a set of sequences witnessing $R$'s Cantor-Bendixson rank $\rk(R)$ that 
consists entirely of $\mu$-randoms. In fact, all of these sequences inductively have the same property.

This shows that the third condition in Theorem~\ref{thm-main} is a very strong property and that, for random degrees, the theorem is a significant strengthening of the result of Downey, Wu, and Yang~\cite{DowWuYan15} cited above.

\subsection{Computable ordinals}

In the proof of Theorem \ref{thm-main}, we will make use of Kleene's system of notations for computable ordinals. Recall that the set of ordinal notations, Kleene's~$\O$, is defined in terms of a function $|\cdot|_\O$ mapping each $a\in\O$ to an ordinal.  We define $|\cdot|_\O$ and a partial order $<_\O$ on $\O$ as follows (see, for instance, Ask and Knight~\cite{AshKni00}). First, $|1|_\O=0$.  Next, if $|a|_\O=\alpha$, then~$|2^a|_\O=\alpha+1$.  Then we define $b<_\O 2^a$ to hold if and only if $b<_\O a$ or $b=a$.  Lastly, for a limit ordinal~$\alpha$, we assign it
the notation $3^e\cdot5$ for any $e\in\omega$ such that $\phi_e$ is a total computable function from~$\omega$ to~$\O$,
\[
\phi_e(0)<_\O\phi_e(1)<_\O\phi_e(2)<_\O\dotsc,
\]
and, setting $\alpha_n=|\phi_e(n)|_\O$, $\alpha=\sup \alpha_n$. Lastly, let $b <_\O 3^e\cdot 5$ if $b<_\O\phi_e(n)$ for some $n\in\omega$.

\section{Special cases of Theorem \ref{thm-main}}\label{sec_special_case}

We study special cases of Theorem~\ref{thm-main} to illustrate the main ingredients of its full proof.

\subsection{The case $\alpha=1$}
We modify a construction of Porter~\cite{Por15}.
\begin{theorem}\label{thm-rk1}
 For each $\Delta^0_2$ random degree $\mathbf{r}$ there is a countably supported computable measure $\mu$ and a sequence $R\in\mathbf{r}$ such that
\begin{itemize}
\item[(i)] $R$ is random with respect to $\mu$, 
\item[(ii)] $R$ has Cantor-Bendixson rank $1$, and
\item[(iii)] the support of $\mu$ is rank-faithful.
\end{itemize}
\end{theorem}

\begin{proof}
Let $A\in\mathbf{r}$ be random and let $(A_s)_{s\in\omega}$ be a $\Delta^0_2$-approximation of $A$. Without loss of generality let 
$A_s\subseteq [0,s)$ for every~$s\in\omega$, and thus $A_s\neq A_{s+1}$ for every~$s\in\omega$.  For each~$X\in\cs$, we recursively define an $X$-computable sequence $(s_i^X)_{i\geq 1}$ of integers by
\begin{itemize}
\item[--] letting $s^X_1$ be the least $s$ such that $X\uh 1=A_s\uh 1$, and
\smallskip

\item[--] for all $n \geq 1$, by letting $s^X_{n+1}$ be the least $s>s_n^X$ such that ${X\uh (n+1)=A_s\uh (n+1)}$.
\end{itemize}
Note that we allow for the possibility that $s_n^X$ is undefined for some $n\in\omega$, in which case it follows that $s_m^X$ is also undefined for all $m>n$.
Using the sequence~$(s_i^X)_{i\geq 1}$, we next define a sequence of finite strings $\sigma^X_1,\sigma^X_2,\dotsc$ such that for each $i\geq 0$, $\sigma_i$ is either finite in length or undefined.  For~$n\geq 0$, $\sigma^X_n$ is defined via
\[
\sigma^X_n=
	\left\{
		\begin{array}{ll}
			{1^{s_n^X}}^\smallfrown0 & \mbox{if } s_n^X \text{ is defined}, \\
			\mbox{undefined} & \mbox{otherwise}.
		\end{array}
	\right.
\]
We then define $\Phi_A(X)$ by
\[
\Phi_A(X)=
	\left\{
		\begin{array}{ll}
			\sigma_1^X\,\sigma_2^X\cdots & \mbox{if } \sigma_i ^X\mbox{ is defined for every } i\geq 0, \\
			\sigma_1^X\,\sigma_2^X\cdots\sigma_{i-1}^X1^\omega & \mbox{if $i$ is least such that } \sigma_i^X \mbox{ is undefined.}
		\end{array}
	\right.
\]

\noindent It is clear by the definition that $\Phi_A$ is a total Turing functional.   It follows that $\P=\Phi_A(\cs)$ is a $\Pi^0_1$ class.  Let $R=\Phi_A(A)$.   Setting $\mu=\lambda_{\Phi_A}$, since $A\in\MLR$, it follows from the preservation of randomness that $R=\Phi_A(A)\in\MLR_\mu$.  We now verify several claims.

\medskip

\noindent {\em Claim 1.}  $R\in\mathbf{r}$.

\smallskip

\noindent {\em Proof.} Clearly $R=\Phi_A(A)\leq_{\mathrm{T}} A$.  Moreover, $A\leq_{\mathrm{T}} (s^A_n)_{n\in\omega}  \leq_{\mathrm{T}} \Phi_A(A)$.  Thus we have~$R\equiv_{\mathrm{T}} A$.\claimqed

\medskip

\noindent {\em Claim 2.}  If $X\neq A$, then there is a least $n\in\omega$ such that $\sigma_n^X$ is undefined, and thus $\Phi_A(X)=\sigma_1^X\,\sigma_2^X\,\cdots\sigma_{n-1}^X1^\omega$.

\smallskip

\noindent {\em Proof.} Immediate from the definitions.\claimqed

\medskip

\noindent {\em Claim 3.}  If $X\neq A$, then $\Phi_A(X)$ is isolated in $\P$.

\nopagebreak\smallskip\nopagebreak

\noindent {\em Proof.} First observe that for $Y,Z\in\cs$, the property $Y\uh n=Z\uh n$ holds if and only if for~all~$i< n$ either $\sigma_i^Y=\sigma_i^Z$  or both $\sigma_i^Y$ and $\sigma_i^Z$ are undefined.  Now given $X\neq A$, by Claim 2 we have $\Phi_A(X)=\sigma_1^X\,\sigma_2^X\,\cdots\sigma_{n-1}^X1^\omega$ for some~$n\in\omega$.  Suppose that $\Phi_A(X)$ is not isolated in~$\P$.  Then there is an infinite sequence  $(Y_i)_{i\in\omega}$ of sequences such that, for every~${i\in\omega}$, $\Phi_A(Y_i)$~branches off of $\Phi_A(X)$ above $\sigma_1^X\,\sigma_2^X\,\cdots\sigma_{n-1}^X$ after agreeing with $\Phi_A(X)$ on an initial segment whose length is different for each $i\in\omega$.  
Note that for each such $Y_i$, if $Y_i(n)=X(n)$, then since $Y_i\uh n=X\uh n$ and $\sigma_n^X$ is undefined, we have that $\sigma_n^{Y_i}$ is undefined and thus $\Phi_A(Y_i)=\Phi_A(X)$.  It follows that we must have $Y_i(n)\neq X(n)$ for all $i\in\omega$.  However, it then follows that for all $i\neq j$ we have $Y_i\uh (n+1)=Y_j\uh (n+1)$, and thus either $\sigma_n^{Y_i}$ is undefined for all~$i$, in which case $\Phi_A(Y_i)=\Phi_A(X)$ for every $i\in\omega$, which is impossible, or $\sigma_n^{Y_{i\vphantom{j}}}=\sigma_n^{Y_{j\vphantom{i}}}$ for all~$i\neq j$.  But this latter condition implies that every $\Phi_A(Y_i)$ branches off of~$\Phi_A(X)$ at exactly the same place, which contradicts our choice of the sequence $(Y_i)_{i\in\omega}$.  Thus $\Phi_A(X)$ must be isolated in $\P$.\claimqed

\medskip

\noindent {\em Claim 4.}  $\rk_\P(R)=1$.

\smallskip

\noindent {\em Proof.} Let $X_n=(A\uh n)^\frown 0^\omega$ for all $n\in\omega$.  Clearly $X_n\neq A$, as $A$ is random, and so by Claim~2, $\Phi_A(X_n)\neq \Phi_A(A)=R$. By Claim~3, $\Phi_A(X_n)$ is isolated in $\P$ for each $n\in\omega$.  In addition, for $n\in\omega$, since $X_n\uh n=A\uh n$, we have ${\sigma^A_1\,\sigma^A_2\dotsc\sigma^A_n\prec\Phi_A(X_n)}$.  So there must be infinitely many $i\neq j$ such that $\Phi_A(X_i)\neq \Phi_A(X_j)$.  Thus, there is a subsequence of distinct sequences $(\Phi_A(X_{n_k}))_{k\in\omega}$ that branch off of $R$ in $\P$, from which the Claim follows. \claimqed

\medskip

By Claim 1, $R$ is non-computable.  Thus there is no $\Pi^0_1$ class $\Q$ in which $R$ is isolated, which implies that $\rk(R)=\rk_\P(R)=1$.   It is now immediate that $\P$ is rank-faithful. 

Finally, for each $Y\in \P$ such that $Y\neq \Phi_A(A)$, $Y=\Phi_A(X)$ for some $X\neq A$, hence $Y=\Phi_A(X)=\sigma_1^X\,\sigma_2^X\,\cdots\sigma_{n-1}^X1^\omega$ by Claim 2.  It follows that $\P$ is countable and $\mu$ is countably supported. 
\end{proof}

\subsection{The case $\alpha<\omega$}  The move from $\alpha=1$ to an arbitrary positive integer requires several new ideas.  We review some definitions; for more details, see, for example, Odifreddi~\cite{MR1718169}.

Given ordinals $\alpha$ and $\beta$, the \emph{Hessenberg sum} of $\alpha$ and $\beta$, written $\alpha\oplus\beta$, is defined as follows.  First, let $\alpha=\omega^{\gamma_1}a_1+\omega^{\gamma_2}a_2+\cdots+\omega^{\gamma_k}a_k$ and ${\beta=\omega^{\gamma_1}b_1+\omega^{\gamma_2}b_2+\cdots+\omega^{\gamma_k}b_k}$ be the Cantor normal forms of $\alpha$ and $\beta$ (where we allow the $a_i$ and $b_i$ to be~$0$ for~$1 \leq i \leq k$).  Then
\[
\alpha\oplus\beta=\omega^{\gamma_1}(a_1+b_1)+\omega^{\gamma_2}(a_2+b_2)+\cdots+\omega^{\gamma_k}(a_k+b_k).
\]
Note that if $\alpha$ and $\beta$ are both finite, then $\oplus$ is just ordinary addition.  Next, for $\Pi^0_1$ classes~$\P$ and~$\Q$, we define the product  of $\P$ and $\Q$ to be 
\[
\P\otimes \Q=\{X\oplus Y\colon X\in\P\;\&\;Y\in\Q\}.
\]  

We will use the following two results of Owings.

\begin{theorem}[Owings~\cite{Owi97}]\label{thm-owings1}
For $X,Y\in\cs$ and $\Pi^0_1$ classes $\P$ and $\Q$,
\[
\rk_{\P\otimes \Q}(X\oplus Y)=\rk_\P(X)\oplus \rk_Q(Y).
\]
\end{theorem}

\begin{theorem}[Owings~\cite{Owi97}]\label{thm-owings2}
For $X,Y\in\cs$, if $\rk(X\oplus Y)=\rk(X)$, then \[Y\leq_{\mathrm{T}} X.\]
\end{theorem}

We now prove the result for finite ordinals $\alpha$.

\begin{theorem}\label{thm-rk2}
 For each $\Delta^0_2$ random degree $\mathbf{r}$ and for all integers $n\geq 1$, there is a countably supported computable measure $\mu$ and a sequence $R\in\mathbf{r}$ such that
\begin{itemize}
\item[(i)] $R$ is random with respect to $\mu$, 
\item[(ii)] $R$ has Cantor-Bendixson rank $n$, and
\item[(iii)] the support of $\mu$ is rank-faithful.
\end{itemize}
\end{theorem}

\begin{proof} We proceed by induction on $n$.  In particular, for each $n\geq 1$ and each $\Delta^0_2$ random sequence $A$, we will inductively define a total Turing functional $\Psi^n_A$ such that
\begin{itemize}
\item[(a)] $\Psi^n_A(A)$ and $\lambda_{\Psi^n_A}$ satisfy the conditions (i)-(iii) of the theorem for $n$;
\item[(b)] $\Psi^n_A(X)\leq_\mathrm{T}A$ for every $X\in\cs$;
\item[(c)]$\Psi^n_A(A)\equiv_\mathrm{T}A$;
\item[(d)] $\rk_{\Psi^n_A(\cs)}(\Psi^n_A(A))=n$; 
\item[(e)] for every $X\neq A$, $\rk_{\Psi_A^n(\cs)}(\Psi^n_A(X))<n$; and 
\item[(f)] $\Psi^n_A(\cs)$ is rank-faithful.
\end{itemize}
In particular, $\Psi^n_A$ is defined as follows:  Splitting $A$ into an $n$-fold join ${A=\bigoplus_{i=0}^{n-1}A_i}$ we split an input sequence $X$ into $n$ sequences $X_0,\dotsc, X_{n-1}$ so that $X=\bigoplus_{i=0}^{n-1} X_i$ and let
\[
\Psi^n_A(X)=\bigoplus_{i=0}^{n-1}\Phi_{A_i}(X_i),
\]
where, for $0\leq i\leq n-1$, $\Phi_{A_i}$ is defined as in the proof of Theorem \ref{thm-rk1}.  

\medskip

\noindent {\em Base case} $(n=1)$:  This is established by Theorem \ref{thm-rk1}. Moreover, for each $\Delta^0_2$ random sequence $A$, defining $\Psi^0_A$ to be the functional $\Phi_A$ as in the proof of Theorem \ref{thm-rk1}, one can verify that the above conditions (a)-(f) hold.

\medskip

\noindent {\em Inductive step}:  Suppose that for a fixed $n\geq 1$, we have shown that for every $\Delta^0_2$ random degree $\mathbf{r}$ and each random $D\in\mathbf{r}$, the total Turing functional $\Psi^n_D$ as defined above satisfies the conditions (a)-(f).

Let $A\in\mathbf{r}$ be random, and let us consider $A=\bigoplus_{i=0}^nA_i$.  Setting ${A_{<n}=\bigoplus_{i=0}^{n-1}A_i}$, we have $A=A_{<n}\oplus A_n$ by our definition of the $n$-fold join given in Section \ref{sec_background}.  Note that, by van Lambalgen's theorem, $A_{<n}$ and $A_n$ are relatively random, and clearly both are $\Delta^0_2$.

As the Turing degree of $A_{<n}$ is random and $\Delta^0_2$, by the inductive hypothesis, the functional
$
\Psi^n_{A_{<n}}(X)=\bigoplus_{i=0}^{n-1}\Phi_{A_i}(X_i)
$,
where $X=\bigoplus_{i=0}^{n-1}X_i$, satisfies the conditions (a)-(f) above.

We define a new Turing functional $\Psi^{n+1}_A$ as follows:  For $X\in\cs$, splitting up~$X$ as $X=\bigoplus_{i=0}^nX_i$ and setting $X_{<n}=\bigoplus_{i=0}^{n-1}X_i$, we define
\[
\Psi^{n+1}_A(X)=\Psi^n_{A_{<n}}(X_{<n})\oplus\Phi_{A_n}(X_n).
\]  
Let $R=\Psi^{n+1}_A(A)$.  We have  $\Psi^n_{A_{<n}}(X_{<n})\leq_{\mathrm{T}} A_{<n}$ and $\Phi_{A_n}(X_n)\leq_{\mathrm{T}} A_n$
 by condition (b) of the inductive hypothesis.  Thus $\Psi^{n+1}_A(X)\leq_\mathrm{T} A$.  In addition, we have  $\Psi^n_{A_{<n}}(A_{<n})\equiv_{\mathrm{T}} A_{<n}$ and $\Phi_{A_n}(A_n)\equiv_{\mathrm{T}} A_n$
 by condition (c) of the inductive hypothesis, and hence $\Psi^{n+1}_A(A)\equiv_\mathrm{T} A$.  In particular, we have $R\equiv_{\mathrm{T}} A$ and so~$R\in\mathbf{r}$.  In addition, since $\Psi^n_{A_{<n}}$ and $\Phi_{A_n}$ are total, it follows that $\Psi^{n+1}_A$ is total, and so $\P=\Psi^{n+1}_A(\cs)$ is a $\Pi^0_1$ class.  In addition, let $\P_0=\Psi^n_{A_{<n}}(\cs)$ and $\P_1=\Phi_{A_n}(\cs)$, so that~$\P=\P_0\otimes\P_1$.  Let~$\mu=\lambda_{\Psi^{n+1}_A}$; due to our choice of $A$ as random, it follows by preservation of randomness that $R \in \MLR_\mu$.

To determine the ranks of elements of $\P$ we will use Theorem \ref{thm-owings1}.  We first prove a claim.

\medskip

\begin{samepage}
\noindent {\em Claim.} The following statements hold.
\begin{itemize}
\item[($C_1$)] If either $X_{<n}\neq A_{<n}$ or $X_n\neq A_n$, then $\rk_\P(\Psi^{n+1}_A(X))\leq n$.
\item[($C_2$)] If $X=A$, then $\rk_\P(\Psi^{n+1}_A(X))=n+1$.
\end{itemize}
\end{samepage}

\smallskip

\noindent {\em Proof.} ($C_1$) We consider two cases.

\medskip

\noindent {\em Case 1:}  $X_{<n}\neq A_{<n}$.  By the inductive hypothesis, specifically the condition (d) above applied to the functional $\Psi^n_{A_{<n}}$, we have $\rk_{\P_0}(\Psi^n_{A_{<n}}(X_{<n}))<n$.  In addition, from the proof of Theorem \ref{thm-rk1} we have $\rk_{\P_1}(\Phi_{A_n}(X_n))\leq 1$. Then by Theorem \ref{thm-owings1} we have
\begin{equation*}
\begin{split}
\rk_\P(\Psi^{n+1}_A(X))&=\rk_{\P_0\otimes\P_1}(\Psi^n_{A_{<n}}(X_{<n})\oplus\Phi_{A_n}(X_n))\\
&=\rk_{\P_0}(\Psi^n_{A_{<n}}(X_{<n}))\oplus\rk_{\P_1}(\Phi_{A_n}(X_n))\\
&<n+1.
\end{split}
\end{equation*}

\medskip

\noindent {\em Case 2:}  $X_n\neq A_n$.  Then by the proof of Theorem \ref{thm-rk1}, we have ${\rk_{\P_1}(\Phi_{A_n}(X_n))=0}$.  By the inductive hypothesis, we have ${\rk_{\P_0}(\Psi^n_{A_{<n}}(X_{<n}))\leq n}$, and so
\[
\rk_\P(\Psi^{n+1}_A(X))=\rk_{\P_0}(\Psi^n_{A_{<n}}(X_{<n}))\oplus\rk_{\P_1}(\Phi_{A_n}(X_n))\leq n+0=n.
\]

\noindent ($C_2$) By the inductive hypothesis, $\rk_{\P_0}(\Psi^n_{A_{<n}}(X_{<n}))=n$.  Moreover, by Theorem \ref{thm-rk1}, $\rk_{\P_n}(\Phi_{A_n}(A_n))=1$.  Thus by Theorem \ref{thm-owings1} we have
\[
\rk_\P(\Psi^{n+1}_A(A))=\rk_{\P_0}(\Psi^n_{A_{<n}}(X_{<n}))\oplus\rk_{\P_1}(\Phi_{A_n}(A_n))=n+1.
\]
This completes the proof of the Claim.\claimqed

\medskip

We now verify that $\P$ is rank-faithful.  Let $Z\in\P$, and let $X\in\cs$ satisfy $\Psi^{n+1}_A(X)=Z$.  By the inductive hypothesis, $\P_0$ and $\P_1$ are rank-faithful, and~so
\begin{equation}\label{eqeqeq}\begin{split}
\rk_\P(\Psi^{n+1}_A(X))&=\rk_{\P_0}(\Psi^n_{A_{<n}}(X_{<n}))\oplus\rk_{\P_1}(\Phi_{A_n}(X_n))\\
&=\rk(\Psi^n_{A_{<n}}(X_{<n}))\oplus\rk(\Phi_{A_n}(X_n)).\end{split}
\end{equation}
By Lemma \ref{lem-rk}, 
\begin{equation}\label{eqeqeqeq}
\rk(\Psi^{n+1}_A(X))\geq\max\{\rk(\Psi^n_{A_{<n}}(X_{<n})),\rk(\Phi_{A_n}(X_n)) \}.
\end{equation}
 By the proof of Theorem \ref{thm-rk1}, $\rk(\Phi_{A_n}(X_n))\leq1$.  We now consider two cases.  
 
 \medskip

  \noindent {\em Case 1:} $\rk(\Phi_{A_n}(X_n))=0$.  Then by Equation (\ref{eqeqeq}), \[{\rk_\P(\Psi^{n+1}_A(X))=\rk(\Psi^n_{A_{<n}}(X_{<n}))},\] and by Equation (\ref{eqeqeqeq}), \[\rk(\Psi^{n+1}_A(X))\geq \rk(\Psi^n_{A_{<n}}(X_{<n})).\]  Thus \[{\rk(\Psi^{n+1}_A(X))=\rk_\P(\Psi^{n+1}_A(X))}.\]
  
  \medskip

\noindent {\em Case 2:}  $\rk(\Phi_{A_n}(X_n))=1$.  By the proof of Theorem \ref{thm-rk1}, $X_n=A_n$.  By Equations (\ref{eqeqeq}) and~(\ref{eqeqeqeq}), we have
\[
\rk(\Psi^n_{A_{<n}}(X_{<n}))\leq\rk(\Psi^{n+1}_A(X))\leq \rk(\Psi^n_{A_{<n}}(X_{<n}))+1=\rk_\P(\Psi^{n+1}_A(X)).
\]
Suppose for the sake of contradiction that $\rk(\Psi^{n+1}_A(X))=\rk(\Psi^n_{A_{<n}}(X_{<n}))$. Then we have 
$\rk\bigl(\Psi^n_{A_{<n}}(X_{<n})\oplus\Phi_{A_n}(X_n)\bigr)=\rk(\Psi^n_{A_{<n}}(X_{<n}))$,
and so, by Theorem \ref{thm-owings2}, ${\Phi_{A_n}(X_n)\leq_{\mathrm{T}}\Psi^n_{A_{<n}}(X_{<n})}$.  By condition (b) of the inductive hypothesis, $\Psi^n_{A_{<n}}(X_{<n})\leq_{\mathrm{T}}A_{<n}$, and thus
\[
A_n\equiv_{\mathrm{T}}\Phi_{A_n}(A_n)=\Phi_{A_n}(X_n) \leq_{\mathrm{T}}\Psi^n_{A_{<n}}(X_{<n})\leq A_{<n},
\]
which contradicts the fact that $A_n$ and $A_{<n}$ are relatively random.  Consequently, \[{\rk(\Psi^{n+1}_A(X))=\rk_\P(\Psi^{n+1}_A(X))},\]
and $\P$ is rank-faithful.
\end{proof}

\subsection{The case $\alpha=\omega$}

Can the above argument be extended to cover the case~$\alpha=\omega$?   One possible strategy would be to split $A=\bigoplus_{i\in\omega} A_i$ into countably many sequences $(A_i)_{i\in\omega}$ and then interleave the corresponding countably many  functionals $(\Phi_{A_i})_{i\in\omega}$, yielding a sequence of the form $\Psi(X)=\bigoplus_{i\in\omega} \Phi_{A_i}(X)$.
The problem with this approach is that the image of $A$ under $\Psi$ would not be ranked, as there are continuum many ways for a sequence to have the wrong information about at least one subsequence $A_i$ of $A$.  In particular, for each~${S\subseteq\omega}$, there is some $Y^S = \bigoplus_{i\in\omega} Y^S_i \in\cs$ such that $Y^S_i\neq A_i$ if and only if $i\in S$, so that $(\Psi(Y^S))_{S\subseteq\cs}$ would yield an uncountable subset of the range of $\Psi$.

The solution we adopt here is to use a dynamic process for interleaving the computations of
potentially an infinite number of sequences of the form $\Phi_{A_0}(X_0)$, $\Phi_{A_1}(X_1),\Phi_{A_2}(X_2),\dotsc$ for $X=\bigoplus_{i\in\omega}X_i$.  More specifically, if $X\neq A$, then we only end up interleaving finitely many such sequences, while if $X=A$, then all such sequences are interleaved.

The general idea is as follows.  Given a uniformly computable sequence of total Turing functionals $\Phi_0,\Phi_1,\Phi_2,\dotsc$, we define a single procedure $\Xi$, which we will refer to as the \emph{dynamic join functional associated to} $(\Phi_i)_{i\in\omega}$, such that, for $X=\bigoplus_{i\in\omega}X_i$, $\Xi(X)$ is obtained by interleaving the sequences $\Phi_0(X_0),\Phi_1(X_1)$, $\Phi_2(X_2),\dotsc$ (or at least of finitely many such sequences, depending on $X$).  The computation of $\Xi(X)$ proceeds in phases, where in Phase~$n$, certain initial segments of the outputs of $\Phi_0(X_0),\Phi_1(X_1),\dotsc,\Phi_n(X_n)$ are interleaved.

If on some input $X$ Phase~$n$ is successfully completed, then the result of that phase will be a finite string $\tau_n^X$ of output bits.  If, for a given input $X$, every phase is successfully completed, there will be a sequence $(\tau^X_i)_{i\in\omega}$ such that
$
\Xi(X)={\tau^X_0}\,{\tau^X_1}\,{\tau^X_2}\dotsc
$

Moreover, if Phase~$n$ is successfully completed, we will say that the computation~$\Phi_{n+1}(X_{n+1})$ has been \emph{activated} as we only begin  to incorporate $\Phi_{n+1}(X_{n+1})$  into the output sequence $\Xi(X)$ in Phase~$n+1$. For consistency, we consider $\Phi_0(X_0)$ activated unconditionally immediately at the start of the computation of~$\Xi(X)$.

The sets $I_k=\{n\in\omega\colon n\equiv 2^{k}-1\pmod {2^{k+1}}\}$ for $k\in\omega$ will play a role in determining into which locations of the output sequence the bits of $\Phi_0(X_0),\Phi_1(X_1),\dotsc$ are copied.  Thus,
\begin{itemize}
\item[--] $I_0=\{n\in\omega\colon n\equiv 0\pmod 2\}$,
\item[--] $I_1=\{n\in\omega\colon n\equiv 1\pmod 4\}$,
\item[--] $I_2=\{n\in\omega\colon n\equiv 3\pmod 8\}$,
\end{itemize}
and so on.   It is easy to see that $(I_i)_{i\in\omega}$ forms  a partition of $\omega$.  We also let $J_0=\omega$ and 
$
J_{n+1}=\overline{I_0 \cup \dots \cup I_{n}},
$ for all $n \in \omega$.
We now describe the phases of the computation of~$\Xi(X)$.

\medskip

 \noindent {\it Phase 0:}  Copy the bits of $\Phi_0(X_0)$ until the first 0 is copied, in which case the phase is successfully completed with the output $\tau_0^X$ and the computation of~$\Phi_1(X_1)$ is activated.  If no such 0 is copied, then copy the bits of $\Phi_0(X_0)$ forever.

\smallskip

 \noindent {\it Phase $n$ $(n\geq 1)$:}   Assume that, in previous phases, we have started  to copy bits from $\Phi_0(X_0),\Phi_1(X_1),\dotsc,\Phi_{n-2}(X_{n-2}),\Phi_{n-1}(X_{n-1})$ into locations in the sets~$I_0, I_1,\dotsc,$ $I_{n-2}$, $ J_{n-1}$ (or simply $J_0$ in the case that $n=1$), resulting in the output $\tau_0^X\,\tau_1^X\cdots\tau_{n-1}^X$ at the end of Phase~$n-1$ as well as the activation of~$\Phi_n(X_n)$.  (a) For each $i\leq n-1$, in the places in $I_i$ that are greater than or equal to~$|\tau_0^X\,\tau_1^X\cdots\tau_{n-1}^X|$, copy the bits of $\Phi_i(X_i)$  that come after those copied in the previous phase and (b) in the locations with indices in $J_n$ greater than or equal to~$|\tau_0^X\,\tau_1^X\cdots\tau_{n-1}^X|$, begin to copy the bits of $\Phi_n(X_n)$, continuing (a) and (b) until a new~$0$ is copied from each of $\Phi_0(X_0),\Phi_1(X_1),\dotsc, \Phi_{n-1}(X_{n-1})$ (beyond those copied in the previous phases) and an initial~$0$ is copied from $\Phi_n(X_n)$.  In this case, the phase is successfully completed with the output $\tau_0^X\,\tau_1^X\cdots\tau_n^X$ and the computation $\Phi_{n+1}(X_{n+1})$ is activated.  If it is not the case that all of $\Phi_0(X_0),\Phi_1(X_1),\dotsc \Phi_n(X_n)$ produce the desired $0$s in this phase, then forever continue copying the bits of $\Phi_0(X_0),\Phi_1(X_1),\dotsc, \Phi_n(X_n)$ in the respective locations described above.

\medskip

This completes the definition of $\Xi$; it is not difficult to verify that $\Xi$~is total. In the case that every phase is successfully completed in the computation of $\Xi(X)$, we have that for each $k\in\omega$ there are numbers $c_k,d_k\in\omega$ (dependent upon $X$) such that
\begin{equation}\label{eq1}
	\Phi_k(X_k)(i+c_k)=\Xi(X)\uh_{I_k}(i+d_k)
\end{equation}
for all $i\in\omega$. To see this, let $c_k$ be the number of bits of $\Phi_k(X_k)$ that were output during Phase~$k$. Then $\Phi_k(X_k)(c_k)$ is the first bit of $\Phi_k(X_k)$ that is output during the Phase~$k+1$; let $d_k$ be location in $I_k$ where it is written. Then it is immediate that Equation~(\ref{eq1}) holds for~$i=0$. Since the computation is now in Phase~$k+1$, by construction all further bits of $\Phi_k(X_k)$ and no further bits from any other functionals are written into $I_k$. Then Equation~(\ref{eq1}) holds for arbitrary $i \in \omega$. In particular, we have $\Phi_k(X_k)\leq_{\mathrm{tt}}\Xi(X)$ for all $k \in \omega$.  

If Phase~$n$ is the last phase that is successfully completed, 
then, by the same argument as above, for all $k\leq n$ there are $(c_k,d_k)$ such that Equation~(\ref{eq1}) holds. Furthermore, 
there is some $d_{n+1}\in\omega$ such that, for all $i\in\omega$,
\begin{equation}\label{eq2}
\Phi_{n+1}(X_{n+1})(i)=\Xi(X)\uh_{J_{n+1}}(i+d_{n+1}).
\end{equation}
To see this, let $d_{n+1}$ be the location in $J_{n+1}$ of the first bit of $\Phi_{n+1}(X_{n+1})$; this bit is output in the Phase~$n+1$. But as that phase is never successfully completed, by construction all further bits of $\Phi_{n+1}(X_{n+1})$ and no further bits from any other functionals are written into~$J_{n+1}$. Then Equation~(\ref{eq2}) holds for all $i \in \omega$, and so $\Phi_k(X_k)\leq_{\mathrm{tt}}\Xi(X)$ for all $k  \leq n+1$. 

\medskip

We can now prove the $\alpha=\omega$ case of Theorem \ref{thm-main}.

\begin{samepage}
\begin{theorem}\label{thm-omegacase}
 For each $\Delta^0_2$ random degree $\mathbf{r}$ there is a countably supported computable measure $\mu$ and a sequence $R\in\mathbf{r}$ such that
\begin{itemize}
\item[(i)] $R$ is random with respect to $\mu$, 
\item[(ii)] $R$ has Cantor-Bendixson rank $\omega$, and
\item[(iii)] the support of $\mu$ is rank-faithful.
\end{itemize}
\end{theorem}
\end{samepage}

\begin{proof}

Let $A\in\mathbf{r}$ be random and let $A=\bigoplus_{i\in\omega}A_i$.  For $i\in\omega$, let $\Phi_{A_i}$ be defined in terms of some $\Delta^0_2$-approximation of $A_i$ as above, let $\Xi$ be the dynamic join functional associated to the sequence $(\Phi_{A_i})_{i\in\omega}$, and set $\P=\Xi(\cs)$, $R=\Xi(A)$, and  $\mu=\lambda_{\Xi}$.  Since $A\in\MLR$ and $\Xi$~is total, the preservation of randomness  implies $R \in \MLR_\mu$.   We verify a series of claims.  

\medskip

\noindent {\em Claim 1.} For $X\neq A$, $\Xi(X)$ only successfully completes finitely many phases.  

\smallskip

\noindent {\em Proof.} 
Suppose that $\Xi(X)$ successfully completes infinitely many phases.  Note that in order for Phase~$n$ to be successfully completed, for $i< n$, $\Phi_{A_i}(X_i)$ must produce $(n+1-i)$ 0s.  Thus, in order for infinitely many phases to be successfully completed, for each $i\in\omega$, $\Phi_{A_i}(X_i)$~must produce infinitely 0s, which occurs only if $X_i=A_i$ by definition of $\Phi_{A_i}$.  Consequently we have $X=A$.
\claimqed
 
 \medskip

\noindent {\em Claim 2.} $\mu$ is countably supported.

\smallskip
 
\noindent {\em Proof.}  For $X\neq A$, since $\Xi(X)$ only successfully completes finitely many phases by Claim 1, $\Xi(X)$~is an interleaving of $\Phi_{A_0}(X_0),\dotsc,\Phi_{A_{j-1}}(X_{j-1})$ for some $j\in\omega$, where for $i< j$, either $\Phi_{A_i}(X_i)=\Phi_{A_i}(A_i)$ or $\Phi_{A_i}(X_i)=\sigma^\smallfrown 1^\omega$ for some $\sigma\in\str$.   This implies that the range of~$\Xi$ is countable, and thus the Claim follows.
\claimqed
 \medskip

\noindent {\em Claim 3.} $\rk(R)=\rk_\P(R)=\omega$.

\smallskip
 
\noindent {\em Proof.} First we show $\rk_\P(R)=\omega$.  For each $X\in\cs$ with $X\neq A$ and~${X=\bigoplus_{i\in\omega}X_i}$, by Claim~1 there is some~$j\in\omega$ such that $\Xi(X)$ only successfully completes Phases $0$ through~${j-1}$.  Then the computation $\Phi_{A_i}(X_i)$ is activated for each $i \leq j$ and the string ${\tau^*=\tau_0^X\tau_2^X\cdots \tau_{j-1}^X}$ is the output obtained at the completion of Phase~$j-1$.  Then there are $\ell_0,\dotsc,\ell_{j-1}\in\omega$ such that $\Xi(X)\uh|\tau^*|$ consists precisely of the bits from
\[
\Phi_{A_0}(X_0)\uh \ell_0,\, \Phi_{A_1}(X_1)\uh \ell_1,\,\dotsc\,,\Phi_{A_{j-1}}(X_{j-1})\uh \ell_{j-1}.
\]
Setting $\ell_{j}=0$, it thus follows from the definition of $\Xi(X)$ and the fact that $\Xi(X)$ does not successfully complete Phase~$j$ that
\begin{equation}\label{eq-claim3}
\Xi(X)\uh_{[|\tau^*|, \,\omega)}=\bigoplus_{i=0}^{j}\Phi_{A_i}(X_i)\uh_{[\ell_i\,,\omega)}.
\end{equation}
(Note that for $Z\in\cs$ and $\ell\in\omega$, $Z\uh_{[\ell,\,\omega)}$ is the tail of $Z$ beginning with its $(\ell+1)$-st bit.)  Let $S$ be the set of $n\in\omega$ such that $X(n)$ is queried in the computation of $\Xi(X)\uh|\tau^*|$.  Then for any $Y\in\cs$, $Y(n)=X(n)$ for all $n\in S$ if and only if $\tau_i^Y=\tau_i^X$ for $i=0,\dotsc,j-1$ and 
\[
\Xi(Y)\uh_{[|\tau^*|,\,\omega)}=\bigoplus_{i=0}^{j}\Phi_{A_i}(Y_i)_{\uh[\ell_i,\,\omega)}.
\]
Setting $\calS_i=\Phi_{A_i}(\cs)\cap\llb\Phi_{A_i}(X_i)\uh \ell_i\rrb$ for $i\leq j$, it follows that
\[
\P\cap\llb\Xi(X)\uh |\tau^*|\rrb=\bigotimes_{i=0}^{j}\calS_i.
\]
Thus 
\[\rk_\P(\Xi(X))=\rk_{\P\cap\llb\Xi(X)\uh |\tau^*|\rrb}(\Xi(X))=\bigoplus_{i=0}^{j}\rk_{\calS_i}(\Phi_{A_i}(X_i))=\#\{i\leq j\colon X_i=A_i\},\]
where the last equality follows from the fact that for $i\leq j$,
\[
\rk_{\calS_i}(\Phi_{A_i}(X_i))=
	\left\{
		\begin{array}{ll}
			0 & \mbox{if } X_i\neq A_i, \\
			1 & \mbox{if } X=A_i.
		\end{array}
	\right.
\]

Now, for $n\geq 0$, we define $B^n$ by first defining a sequence ($B_i^n)_{i\in\omega}$ as follows.
\begin{itemize}
\item[--] $B^n_i=A_i$ for $i < n$; and
\item[--] $B^n_i=0^\omega$ for $i \geq n$.
\end{itemize}
Then we set $B^n=\bigoplus_{i\in\omega}B_i^n$.  
We claim that $\rk_\P(\Xi(B^n))=n$ for every ${n\geq 0}$.  For a fixed~$n\in\omega$, suppose that Phase $j-1$ is the last phase successfully completed in the computation of~$\Xi(B^n)$; clearly $j\geq n$.  Then by the above argument, since~$B^n\neq A$,
\[
\rk_\P(\Xi(B^n))=\#\{i\leq j\colon B^n_i=A_i\}=n.\]
As we have $\Xi(B^n)\geq_{\mathrm{tt}}\bigoplus_{i=0}^{n-1}\Phi_{A_i}(A_i)$, and since $\rk\bigl(\bigoplus_{i=0}^{n-1}\Phi_{A_i}(A_i)\bigr)=n$ by the proof of Theorem~\ref{thm-rk2},  we can apply Lemma \ref{lem-rk} to obtain $\rk(\Xi(B^n))\geq n$.  Thus $\rk(\Xi(B^n))=n$.

As the sequence $(\Xi(B^n))_{n\in\omega}$ converges to $\Xi(A)$, ${\rk_\P(\Xi(A))\geq\sup\rk_\P(\Xi(B^n))}$.  Moreover, for $X\neq A$, by Claim 1 and the argument above, $\rk_\P(\Xi(X))<\omega$, and so we have $\rk_\P(R)=\rk_\P(\Xi(A))=\omega$. 
To see that $\rk(R)=\omega$, observe that since $\Xi(A)\uh_{I_k}\geq_{\mathrm{tt}} \Phi_{A_k}(A_k)$ for every $k\in\omega$, we have
\[
\Xi(A)\geq_{\mathrm{tt}}\bigoplus_{i=0}^{n-1}\Phi_{A_i}(A_i)
\]
for every $n\geq 1$.  Since $\rk\bigl(\bigoplus_{i=0}^{n-1}\Phi_{A_i}(A_i)\bigr)=n$, for every $n$, by Lemma \ref{lem-rk}, $\rk(\Xi(A))\geq n$ for every $n\in\omega$ and hence $\rk(R)=\omega$.\claimqed

\medskip

\noindent {\em Claim 4.} $R \in \mathbf{r}$.

\smallskip

\noindent {\em Proof.}  By Claim 1 and the fact that $\Xi(A)$ successfully completes infinitely many phases, $\Xi^{-1}(\{\Xi(A)\})=\{A\}$.  Thus $\{A\}$ is a $\Pi^0_1$ class relative to $\Xi(A)$.  As $A$ is isolated in this class, we have $A\leq_{\mathrm{T}}\Xi(A)$. Since clearly $R=\Xi(A)\leq_{\mathrm{T}}A$, $R\equiv_{\mathrm{T}}A$.\claimqed

\medskip

\noindent {\em Claim 5.} $\Xi(X)\leq_{\mathrm{T}} A$ for all $X\in\cs$.

\smallskip

\noindent {\em Proof.}  By Claim 4, $\Xi(A)\leq_{\mathrm{T}}A$.  So suppose $X\neq A$.  From Equation (\ref{eq-claim3}) in the proof of Claim 3, $\Xi(X)\equiv_{\mathrm{T}}\bigoplus_{i=0}^{j}\Phi_{A_i}(X_i)$ for some $j\in\omega$.  From the proof of Theorem \ref{thm-rk2}, $A\geq_{\mathrm{T}}\bigoplus_{i=0}^j\Phi_{A_i}(X_i)\geq_{\mathrm{T}}\Xi(X)$.\claimqed

\medskip

\noindent {\em Claim 6.} $\P$ is rank-faithful.

\smallskip

\noindent {\em Proof.} We have already shown that $\rk(R)=\rk_\P(R)$.  For $Z\in\P$ with $Z\neq R$, there is some~$Y\neq A$ such that $Z=\Xi(Y)$.  By Claim 1, there is a $j\in\omega$ such that Phase $j-1$ is the last phase successfully completed in the computation of~$\Xi(Y)$. Then from the proof of Claim 3, we see that both
\begin{equation}\label{eq-rk}
\Xi(Y)\equiv_{\mathrm{tt}}\bigoplus_{i=0}^{j}\Phi_{A_i}(Y_i)
\end{equation}
and that $\rk_\P(\Xi(Y))=\#\{i\leq j\colon Y_i=A_i\}$.
Set $Z^*=\bigoplus_{i=0}^{j}\Phi_{A_i}(Y_i)$ and~${A^*=\bigoplus_{i=0}^{j}A_i}$. Note that $A^*$ is random and $\Delta^0_2$, so that by the same argument as in the proof of Theorem \ref{thm-rk2} we can see that \[\rk_{\Psi_{A^*}^{j}(\cs)}(Z^*)=\#\{i\leq j\colon Y_i=A_i\}\] and that $\Psi_{A^*}^{j}(\cs)$ is rank-faithful. So it follows that  
\[\rk(Z^*)=\#\{i\leq j\colon Y_i=A_i\}.\]  Applying Lemma \ref{lem-rk} to Equation (\ref{eq-rk}) yields 
\[
\rk(\Xi(Y))=\#\{i\leq j\colon Y_i=A_i\}=\rk_\P(\Xi(Y)),
\]
 and so $\P$ is rank-faithful.\claimqed

\medskip

This completes the proof of Theorem~\ref{thm-omegacase}
\end{proof}

\section{Proof of Theorem \ref{thm-main}}\label{sec_full_proof}

When combining the ideas used in the proofs of Theorems \ref{thm-rk1}, \ref{thm-rk2}, and \ref{thm-omegacase} to obtain the full proof of Theorem \ref{thm-main}, there are several issues of uniformity that we have to take into account.  In particular, we will define a functional
$\Theta(e,a,X)$, where $e$ is the index of a $\Delta^0_2$-approximation of a  set, $a$ is the notation of a computable ordinal, and $X\in\cs$.  We will also make use of the fact that there are computable functions $g\colon  \omega\times \{0,1\}\rightarrow \omega$ and $h\colon \omega\times\omega\rightarrow\omega$ such that if $e$ is an index of a $\Delta^0_2$-approximation of a sequence $A\in\omega$, then
\begin{itemize}
\item[--] for $i\in\{0,1\}$, $g(e,i)$ is an index of a $\Delta^0_2$-approximation of $A_i$, where ${A=A_0\oplus A_1}$,
\item[--] for $n\in\omega$, $h(e,n)$ is an index of a $\Delta^0_2$-approximation of $A_n$, where ${A=\bigoplus_{i\in\omega}A_i}$.
\end{itemize}

\begin{proof}[Proof of Theorem~\ref{thm-main}]
Let $A\in\mathbf{r}$ be a $\Delta^0_2$ random sequence and let $j$ be an index for some $\Delta^0_2$-approximation $(A_s)_{s\in\omega}$ of $A$.
We proceed by effective transfinite induction.  \\

\smallskip

\noindent {\em Case 1:}  $\alpha=1$.  This is Theorem \ref{thm-rk1}.\\

\smallskip

\noindent {\em Case 2:} $\alpha=\beta+1$.  Assume that for each $\Delta^0_2$ sequence $B\in\MLR$, each index~$k$ of a $\Delta^0_2$\nobreakdash-approximation of $B$, and for each $c\in\O$ such that $|c|_\O<\alpha$, we have defined $\Theta$ on all inputs of the form $(k,c,X)$ for arbitrary $X\in \cs$ in such a way that each of the restricted functionals $\Theta(k,c,\cdot)\colon \cs\rightarrow\cs$ is total and such that
\begin{itemize}
\item[(a)] $\Theta(k,c,B)$ and $\lambda_{\Theta(k,c,\cdot)}$ satisfy the conditions (i)-(iii) of the theorem for $|c|_\O$;
\item[(b)] $\Theta(k,c,X)\leq_{\mathrm{T}} B$ for every $X\in\cs$;
\item[(c)] $\Theta(k,c,B)\equiv_{\mathrm{T}} B$;
\item[(d)] $\rk_\P(\Theta(k,c,B))=|c|_{\mathcal{O}}$, where $\P$ is the range of $\Theta(k,c,\cdot)$; 
\item[(e)] for every $X\neq B$, $\rk_\P(\Theta(k,c,X))<|c|_\O$; and
\item[(f)] $\P$ is rank-faithful.
\end{itemize}
Then for $b\in\omega$ with $|b|_{\O}=\beta$, we define 
\begin{samepage}\[
\Theta(j,2^b,X)=\Theta\bigl(g(j,0),b,X_0\bigr)\oplus\Theta\bigl(g(j,1),2,X_1\bigr),
\]
where $X=X_0\oplus X_1$.\end{samepage} Note that $|2^b|_{\O}=\alpha$ and $|2|_{\O}=1$.

Since $g(j,0)$ is an index for a $\Delta^0_2$-approximation of $A_0$, by our inductive hypothesis, $\Theta(g(j,0),b,\cdot)$ induces a computable, countably supported measure $\mu_0$ concentrated on a $\Pi^0_1$~class~$\P_0$ containing the sequence $\Theta(g(j,0),b,A_0)$, which satisfies the conditions of the theorem for the ordinal $\beta$.  In particular, we have ${\rk(\Theta(g(j,0),b,A_0))=\rk_{\P_0}(\Theta(g(j,0),b,A_0))=\beta}$.

Similarly, $g(j,1)$ is an index for a $\Delta^0_2$-approximation of $A_1$, and so by our inductive hypothesis, $\Theta(g(j,1),2,\cdot)$ induces a computable, countably measure $\mu_1$ concentrated on a $\Pi^0_1$~class~$\P_1$ containing the sequence $\Theta(g(j,1),2,A_1)$, which satisfies the conditions of the theorem for the ordinal 1.  
In particular, we have $\rk(\Theta(g(j,1),2,A_1))=\rk_{\P_1}(\Theta(g(j,1),2,A_1))=1$.

Let $\P=\P_0\otimes\P_1$ and let ${\mu=\mu_0\otimes\mu_1}$ be defined by ${\mu(\llb\sigma\oplus\tau\rrb)=\mu_0(\sigma)\cdot\mu_1(\tau)}$. 
We verify that the sequence ${R=\Theta(j,2^b,A)=\Theta(g(j,0),b,A_0)\oplus\Theta(g(j,1),2,A_1)}$ satisfies the conditions of the theorem for $\alpha=\beta+1$.
First, $\mu$ is countably supported because~$\mu_0$ and~$\mu_1$ are countably supported.  Moreover, $R\in\MLR_\mu$ by the preservation of randomness.  By condition~(b) of the inductive hypothesis we have $\Theta(g(j,0),b,X)\leq_{\mathrm{T}}A_0$ and $\Theta(g(j,1),2,X)\leq_{\mathrm{T}}A_1$ for all $X\in\cs$. Thus, for~$X=X_0\oplus X_1\in\cs$,
\[
\Theta(j,2^b,X)=\Theta(g(j,0),b,X_0)\oplus\Theta(g(j,1),2,X_1)\leq_{\mathrm{T}} A_0\oplus A_1=A.  
\]
By condition (c) of the inductive hypothesis we have 
\[
R=\Theta(j,2^b,A)=\Theta(g(j,0),b,A_0)\oplus\Theta(g(j,1),2,A_1)\equiv_{\mathrm{T}} A_0\oplus A_1=A.  
\]
To see that $\rk(R)=\alpha$,  first note that, by Theorem \ref{thm-owings1},
\[
\rk_\P(R)=\rk_{\P_0}(\Theta(g(j,0),b,A_0))\oplus\rk_{\P_1}(\Theta(g(j,1),2,A_1))=\beta+1.
\]
Since $R\geq_{\mathrm{tt}}\Theta(g(j,0),b,A_0)$, by Lemma \ref{lem-rk}, $\rk(R)\geq\rk(\Theta(g(j,0),b,A_0))=\beta$.  In fact, we must even have $\rk(R)=\alpha$, for if $\rk(R)=\beta$, then by Theorem \ref{thm-owings2} we would have ${R\leq_{\mathrm{T}} \Theta(g(j,0),b,A_0)}$, which contradicts the fact that 
\[
\Theta(g(j,1),2,A_1)\equiv_{\mathrm{T}} A_1\not\leq_{\mathrm{T}} A_0\equiv_{\mathrm{T}} \Theta(g(j,0),b,A_0).
\]
So $\P$ is rank-faithful.

\medskip

\noindent {\em Case 3:}  $\alpha$ is a limit ordinal. Then there is some $e\in\omega$ such that $|3\cdot 5^e|_\O=\alpha$.  It follows that $(\phi_e(n))_{n\in\omega}$ is a sequence of notations of computable ordinals $\gamma_n$ such that $\sup_{n\in\omega}\gamma_n=\alpha$. Setting $A=\bigoplus_{i\in\omega}A_i$, it follows from the remarks at the beginning of this section that $h(j,n)$ is an index for a $\Delta^0_2$-approximation of~$A_n$ for each $n\in\omega$.  

Again assume that for each $\Delta^0_2$ sequence $B\in\MLR$, each index $k$ of a $\Delta^0_2$\nobreakdash-approximation of $B$, and for each $d\in\O$ such that $d<_\O 3\cdot 5^e$, we have defined $\Theta$ on all inputs of the form $(k,d,X)$ for arbitrary $X\in \cs$ in such a way that each of the restricted functionals $\Theta(k,d,\cdot)\colon \cs\rightarrow\cs$ is total and 
satisfies the conditions (a)-(f) above.

For each $i\in\omega$, let $\Phi_i=\Theta(h(j,i),\phi_e(i),\cdot)$ and let $\Xi$ be the dynamic join functional associated with the sequence $(\Phi_i)_{i\in\omega}$.  By the inductive hypothesis, for each $i\in\omega$, $\Phi_i$~induces a computable measure $\mu_i$ concentrated on a $\Pi^0_1$ class~$\Q_i$ which contains the sequence~$\Phi_i(A_i)$ and satisfies the conditions of the theorem for $\gamma_i$.

If $\mu$ is the measure induced by $\Xi$, then since $\Xi$ is total and $A\in\MLR$, ${\Xi(A)\in\MLR_\mu}$.  Let~$\P=\Xi(\cs)$.  We verify claims analogous to those in the proof of Theorem \ref{thm-omegacase}.

\medskip

\noindent {\em Claim 1.} For $X\neq A$, $\Xi(X)$ only successfully completes finitely many phases.  

\smallskip

\noindent {\em Proof.} As this depends solely on the general properties of~$\Xi$, no new argument is needed.\claimqed

\medskip

\noindent {\em Claim 2.} $\mu$ is countably supported.

\medskip

\noindent {\em Proof.} As this depends solely on the general properties of~$\Xi$, no new argument is needed.\claimqed

\medskip

\noindent {\em Claim 3.} $\rk_\P(R)=\alpha$.

\nopagebreak\smallskip\nopagebreak

\noindent {\em Proof.} Proceeding as in the proof of Claim 3 inside the proof of Theorem \ref{thm-omegacase}, for each $X\in\cs$ with $X\neq A$ and $X=\bigoplus_{i\in\omega}X_i$ we have
\[
\rk_\P(\Xi(X))=\bigoplus_{\{i\leq j\colon X_i=A_i\}}\gamma_i.
\]
Next, for $n\geq 0$, as before we define ${B^n=\bigoplus_{i\in\omega}B_i^n}$, where
\begin{itemize}\begin{samepage}
\item[--] $B^n_i=A_i$ for $i< n$; and
\item[--] $B^n_i=0^\omega$ for $i\geq n$.\end{samepage} 
\end{itemize}

It is straightforward to verify as above that for every $n$, there is some $j\geq n$ such that $\rk_\P(\Xi(B^n))=\bigoplus_{i=0}^j\gamma_i<\alpha$.
As in the proof of Theorem~\ref{thm-omegacase}, the sequence $(\Xi(B^n))_{n\in\omega}$ converges to $\Xi(A)$, thus \[\rk_\P(\Xi(A))\geq\sup\rk_\P(\Xi(B^n))=\alpha.\] Furthermore, again  as in the proof of Theorem~\ref{thm-omegacase}, for $X\neq A$ we have $\rk_\P(\Xi(X))<\alpha$, and so ${\rk_\P(R)=\rk_\P(\Xi(A))=\alpha}$.

As it holds that $\Xi(A)\uh_{I_k}\geq_{\mathrm{tt}} \Phi_k(A_k)$ for every $k\in\omega$ we have that
\[\Xi(A)\geq_{\mathrm{tt}}\bigoplus_{i=0}^{n-1}\Phi_i(A_i)\]
for every $n\geq 1$.  Since $\rk\bigl(\bigoplus_{i=0}^{n-1}\Phi_{i}(A_i)\bigr)=\bigoplus_{i=0}^{n-1}\gamma_i\geq\gamma_{n-1}$ for every $n$ by the inductive hypothesis,
Lemma \ref{lem-rk} implies $\rk(\Xi(A))\geq \gamma_{n-1}$ for every $n\in\omega$. Thus~${\rk(R)=\alpha}$.\claimqed

\medskip

\noindent {\em Claim 4.} $R \in \mathbf{r}$.

\nopagebreak\smallskip\nopagebreak

\noindent The proof is the same as that of Claim 4 in the proof of Theorem~\ref{thm-omegacase}.

\medskip

\noindent {\em Claim 5.} $\Xi(X)\leq_{\mathrm{T}} A$ for all $X\in\cs$.

\nopagebreak\smallskip\nopagebreak

\noindent The proof is the same as that for Claim 5 in the proof of Theorem~\ref{thm-omegacase}.

\medskip

\noindent {\em Claim 6.} $\P$ is rank-faithful.

\nopagebreak\smallskip\nopagebreak

\noindent {\em Proof.} The proof is very similar to that for Claim 6 in the proof of Theorem~\ref{thm-omegacase}.   We highlight the key differences. First, for any sequence $Y\neq A$ so that Phase~${j-1}$ is the last phase that is successfully completed in the computation of~$\Xi(Y)$, we have 
$
\rk_\P(\Xi(Y))=\bigoplus_{i=0}^{j}\gamma_i
$.
Next, as before,
$
{\Xi(Y)\equiv_{\mathrm{tt}}\bigoplus_{i=0}^{j}\Phi_{A_i}(Y_i)}
$,
and so by Lemma \ref{lem-rk}, $\rk(\Xi(Y))=\rk\bigl(\bigoplus_{i=0}^{j}\Phi_{A_i}(Y_i)\bigr).$
Setting $A^*=\bigoplus_{i=0}^{j}A_i$ as before, by the inductive hypothesis, $\Psi_{A^*}^{j}(\cs)$ is a rank-faithful $\Pi^0_1$~class in which we have 
\[
\rk_{\Psi_{A^*}^{j}(\cs)}\Biggl(\bigoplus_{i=0}^{j}\Phi_{A_i}(Y_i)\Biggr)=\bigoplus_{i=0}^{j}\gamma_i.
\]
Thus, we can conclude from the above that
\[
\rk(\Xi(Y))=\rk\Biggl(\bigoplus_{i=0}^{j}\Phi_{A_i}(Y_i)\Biggr)=\rk_{\Psi_{A^*}^{j}(\cs)}\Biggl(\bigoplus_{i=0}^{j}\Phi_{A_i}(Y_i)\Biggr)=\bigoplus_{i=0}^{j}\gamma_i=\rk_\P(\Xi(Y)),
\] as needed.
\claimqed

\medskip

\noindent This completes the proof of Theorem~\ref{thm-main}.
\end{proof}

\section{Future work}  

Theorem \ref{thm-main} immediately holds for every randomness notion $\mathcal{R}$ with the property that $\mu$-Martin-L\"of randomness implies $\mu$\nobreakdash-$\mathcal{R}$\nobreakdash-randomness; in particular it holds for $\mu$-Schnorr randomness and $\mu$\nobreakdash-computable randomness.  It might be worth exploring in future work whether some analogue of the theorem holds for notions of randomness that are stronger than Martin-L\"of randomness and allow for $\Delta^0_2$ random sequences.

As shown by Cenzer and Remmel~\cite{CenRem98}, for any computable ordinal $\alpha$ that is either $0$ or a limit ordinal and any $n\in\omega$, for any Turing degree $\mathbf{a}$ such that $\mathbf{0}^{(\alpha+2n+1)}\leq\mathbf{a}\leq\mathbf{0}^{(\alpha+2n+2)}$, there is some sequence $B$ of degree $\mathbf{a}$ with $\rk(B)=\alpha+n+1$. 

\begin{question}
Can $B$ be chosen to be proper?
\end{question}

\bibliographystyle{asl}
\bibliography{rankrandomness}

\end{document}